\documentclass{article}
\usepackage{amsfonts}
\usepackage{amsmath}

\newtheorem{theorem}{Theorem}

\newtheorem{corollary}[theorem]{Corollary}

\newtheorem{proposition}[theorem]{Proposition}

\newenvironment{proof}[1][Proof]{\noindent \textbf{#1.} }{\  \rule{0.5em}{0.5em}}
\input{tcilatex}
\begin{document}

\title{\textbf{SURFACE PENCILS IN EUCLIDEAN 4-SPACE }$\mathbb{E}^{4}$\textbf{%
\ }}
\author{ Bet\"{u}l Bulca \  \& Kadri Arslan }
\date{}
\maketitle

\begin{abstract}
In the present paper we study the problem of constructing a family of
surfaces (surface pencils) from a given curve in 4-dimensional Euclidean
space $\mathbb{E}^{4}.$ We have shown that generalized rotation surfaces in $%
\mathbb{E}^{4}$ are the special type of surface pencils. Further, the
curvature properties of these surfaces are investigated. \ Finally, we give
some examples of flat surface pencils in $\mathbb{E}^{4}.$
\end{abstract}

\section{Introduction}

\footnote{%
2000 \textit{Mathematics Subject Classification}. 53C40, 53C42
\par
\textit{Key words and phrases}: Surface pencil, marching-scale function,
Gaussian curvature.}The problem of a constructing a family of surfaces from
given curve is (i.e. surface pencils) important for differential geometry.
In recent years surface pencils in $\mathbb{E}^{3}$ was studied many paper
with respect to the curves family. In 2004 Wang et al. studied the problem
of constructing a surface family from a given spatial geodesic \cite{WTT}.
Further, Lie et al. derived the necessary and sufficient condition for a
given curve to be the line of curvature on a surface \cite{LVZ}. However,
Kasap et al. generalized the marching-scale functions given in \cite{WTT}
and gave a sufficient condition for a given curve to be a geodesic on a
surface \cite{KAO}. Recently, Bayram et al. extend the method given in \cite%
{WTT} to derive the necessary and sufficient condition for a given curve to
be both isoparametric and asymptotic on a parametric surface \cite{BGK}.
Also, Erg\"{u}n et al. considered surface pencil with a common line of
curvature in Minkowski 3-space \cite{EBK}. In 2008, Zhao and Wang proposed a
new method for designing developable surface by constructing a surface
pencil passing through a given curve which is quite in accord with the
practice in industry design and manufacture \cite{ZW}.

In the present paper we extend the surface pencil in $4$-dimensional
Euclidean space $\mathbb{E}^{4}.$ The object of study in this paper is to
extend the correct parametric representation of the surface $M$ for a given
curve $\gamma (s)$ in $4$-dimensional Euclidean space $\mathbb{E}^{4}.$ This
paper consist of 3 sections. The first section is introduction. In the
Section 2 we give some basic concepts of surfaces in $\mathbb{E}^{4}$\ which
are used in the further sections of this paper. In Section $3$, by utilizing
the Frenet frame from differential geometry, we derive necessary and
sufficient condition for the correct representation of the surface patch $%
X(s,t),$ where the parameter $s$ is the arc-length of the curve $\gamma (s).$
The basis idea is to represent $X(s,t)$ as a linear combination of the
vector functions $V_{2}(s)$ and $V_{4}(s)$ which are the normal vector and
second binormal vector of $\gamma (s)$\ respectively. The surface pencil
which we consider in the present paper is parametrized by 
\begin{equation}
X(s,t)=\gamma (s)+A(t)V_{2}(s)+B(t)V_{4}(s),\text{ }t\in J\subset \mathbb{R},
\label{B1}
\end{equation}%
where $A(t)$ and $B(t)$ are differentiable functions which are called
marching-scale functions. Further, we have calculated the Gaussian, normal
and mean curvature of this surface. However, we obtain some equations of the
marching-scale functions $A(t)$ and $B(t)$ for the case $M$ becomes a
Vranceanu surface. Finally, we give some examples of flat surface pencils in 
$\mathbb{E}^{4}$ given with the spacial marching-scale functions.

\section{Basic Concepts}

Let $M$ be a smooth surface in $\mathbb{E}^{4}$ given with the patch $X(u,v)$
: $(u,v)\in D\subset \mathbb{E}^{2}$. The tangent space to $M$ at an
arbitrary point $p=X(u,v)$ of $M$ span $\left \{ X_{u},X_{v}\right \} $. In
the chart $(u,v)$ the coefficients of the first fundamental form of $M$ are
given by 
\begin{equation}
E=\langle X_{u},X_{u}\rangle ,F=\left \langle X_{u},X_{v}\right \rangle
,G=\left \langle X_{v},X_{v}\right \rangle ,  \label{A1}
\end{equation}%
where $\left \langle ,\right \rangle $ is the Euclidean inner product. We
assume that $g=EG-F^{2}\neq 0,$ i.e. the surface patch $X(u,v)$ is regular.

Let $\chi (M)$ and $\chi ^{\perp }(M)$ be the space of the smooth vector
fields tangent to $M$ and the space of the smooth vector fields normal to $M$%
, respectively. Given any local vector fields $X_{_{1}},X_{_{2}}$ tangent to 
$M$ consider the second fundamental map $h:\chi (M)\times \chi
(M)\rightarrow \chi ^{\perp }(M);$%
\begin{equation}
h(X_{i},X_{_{j}})=\widetilde{\nabla }_{X_{_{i}}}X_{_{j}}-\nabla
_{X_{_{i}}}X_{_{j}}\text{ \  \  \ }1\leq i,j\leq 2.  \label{A3}
\end{equation}%
This map is well-defined, symmetric and bilinear.

For any arbitrary orthonormal frame field $\left \{ N_{1},N_{2}\right \} $
of $M$, recall the shape operator $A:\chi ^{\perp }(M)\times \chi
(M)\rightarrow \chi (M);$%
\begin{equation}
A_{N_{i}}X=-(\widetilde{\nabla }_{X_{i}}N_{i})^{T},\text{ \  \  \ }X_{i}\in
\chi (M).  \label{A4}
\end{equation}%
This operator is bilinear, self-adjoint and satisfies the following equation:

\begin{equation}
\left \langle A_{N_{k}}X_{j},X_{i}\right \rangle =\left \langle
h(X_{i},X_{j}),N_{k}\right \rangle =c_{ij}^{k}\text{, }1\leq i,j,k\leq 2
\label{A5}
\end{equation}%
where $c_{ij}^{k}$ are the coefficients of the second fundamental form. The
equation (\ref{A3}) is called Gaussian formula and%
\begin{equation}
h(X_{i},X_{j})=\overset{2}{\underset{k=1}{\sum }}c_{ij}^{k}N_{k},\  \  \  \  \
1\leq i,j\leq 2.  \label{A6}
\end{equation}

The \textit{Gaussian curvature} and \textit{normal curvature} of the surface 
$M$ are given by%
\begin{equation}
K=\frac{1}{W^{2}}\sum%
\limits_{k=1}^{2}(c_{11}^{k}c_{22}^{k}-(c_{12}^{k})^{2}),  \label{A7}
\end{equation}%
and 
\begin{equation}
K_{N}=\frac{1}{W^{2}}\left( E\left(
c_{12}^{1}c_{22}^{2}-c_{12}^{2}c_{22}^{1}\right) -F\left(
c_{11}^{1}c_{22}^{1}-c_{11}^{2}c_{22}^{1}\right) +G\left(
c_{11}^{1}c_{12}^{2}-c_{11}^{2}c_{12}^{1}\right) \right) ,  \label{A8}
\end{equation}%
respectively. Recall that a surface $M$ is said to be \textit{flat (resp.
has flat normal bundle)} if its Gaussian curvature $K$ (resp. normal
curvature $K_{N}$) vanishes identically. Further, the mean curvature vector
of the surface $M$ is defined by%
\begin{equation}
\overrightarrow{H}=\frac{1}{2W^{2}}%
\sum_{k=1}^{2}(c_{11}^{k}G+c_{22}^{k}E-2c_{12}^{k}F)N_{k}.  \label{A9}
\end{equation}

Recall that a surface $M$ is said to be \textit{minimal} if its mean
curvature vector vanishes identically \cite{Ch}.

\section{Surface Pencils in $\mathbb{E}^{4}$}

Let $\gamma =\gamma (s):I\rightarrow \mathbb{E}^{4}$ be a unit speed regular
curve in Euclidean 4-space $\mathbb{E}^{4}$. The corresponding Frenet
formulas have the following form:%
\begin{eqnarray}
\gamma ^{^{\prime }}(s) &=&V_{1}(s),  \notag \\
V_{1}^{^{\prime }}(s) &=&\kappa _{1}(s)V_{2}(s),  \notag \\
V_{2}^{^{\prime }}(s) &=&-\kappa _{1}(s)V_{1}(s)+\kappa _{2}(s)V_{3}(s),
\label{B0} \\
V_{3}^{^{\prime }}(s) &=&-\kappa _{2}(s)V_{2}(s)+\kappa _{3}(s)V_{4}(s), 
\notag \\
V_{4}^{^{\prime }}(s) &=&-\kappa _{3}(s)V_{3}(s),  \notag
\end{eqnarray}%
where $V_{1}(s),V_{2}(s),V_{3}(s),V_{4}(s)$ is the Frenet frame field and $%
\kappa _{1},\kappa _{2}$ and $\kappa _{3}$ are the Frenet curvatures of $\ 
\mathbb{\gamma }(s).$ If the Frenet curvatures are constant then $\mathbb{%
\gamma }(s)$ is called \textit{W-curve }\cite{KL}.

Let $M$ be a local surface given with the regular patch%
\begin{equation*}
X(s,t)=\gamma (s)+A(t)V_{2}(s)+B(t)V_{4}(s),\text{ }t\in J\subset \mathbb{R},
\end{equation*}%
where, $A=A(t),$ and $B=B(t)$ are smooth functions, defined in $J\subset 
\mathbb{R}$ and satisfying 
\begin{eqnarray*}
(1-\kappa _{1}(s)A(t))^{2}+(\kappa _{2}(s)A(t)-\kappa _{3}(s)B(t))^{2} &>&0,
\\
A^{\text{ }\prime }(t)^{2}+B^{\prime }(t)^{2} &>&0.
\end{eqnarray*}

For the sake of simplicity let us denote;%
\begin{eqnarray}
a(s,t) &:&=1-\kappa _{1}(s)A(t),  \label{B2} \\
b(s,t) &:&=\kappa _{2}(s)A(t)-\kappa _{3}(s)B(t).  \notag
\end{eqnarray}

This surface is one-parameter family of plane curves $\alpha (t)=\left(
A(t),B(t)\right) $ lying in the normal plane span $\left \{
V_{2}(s),V_{4}(s)\right \} $ of $\gamma .$\textit{\ }The surface given with
the parametrization (\ref{B1}) is called \textit{pencil surface }in $\mathbb{%
E}^{4}.$ If $\gamma (s)$ is a W-curve then $M$ becomes a \textit{generalized
rotation surface} defined by Ganchev and Milousheva in \cite{GM} and see
also \cite{ABBO}.

We prove the following result.

\begin{proposition}
Let $M$ be a \textit{pencil surface} given by the parametrization (\ref{B1}%
). Then the Gaussian curvature of $M$ is%
\begin{equation}
K=\frac{\left( a^{2}+b^{2}\right) \left( A^{\prime }B^{\prime \prime
}-B^{\prime }A^{\prime \prime }\right) \left \{ A^{\prime }b\kappa
_{3}-B^{\prime }(\kappa _{1}a-\kappa _{2}b)\right \} -\left( (A^{\prime
})^{2}+(B^{\prime })^{2}\right) \left( ab_{t}-ba_{t}\right) ^{2}}{g^{2}}.
\label{B3}
\end{equation}%
where $a(s,t)$ and $b(s,t)$ are smooth functions defined in (\ref{B2}).
\end{proposition}

\begin{proof}
The tangent space of $M$ is spanned by the vector fields%
\begin{eqnarray*}
X_{s} &=&a(s,t)V_{1}(s)+b(s,t)V_{3}(s), \\
X_{t} &=&A^{\prime }(t)V_{2}(s)+B^{\prime }(t)V_{4}(s).
\end{eqnarray*}%
Hence, the coefficients of the first fundamental form of the surface are%
\begin{eqnarray}
E &=&\text{ }\left \langle X_{s},X_{s}\right \rangle \text{ }%
=a(s,t)^{2}+b(s,t)^{2},  \notag \\
F &=&\text{ }\left \langle X_{s},X_{t}\right \rangle \text{ }=0,  \label{B4}
\\
G &=&\text{ }\left \langle X_{t},X_{t}\right \rangle \text{ }=(A^{\prime
}(t))^{2}+(B^{\prime }(t))^{2},  \notag
\end{eqnarray}%
where $\left \langle ,\right \rangle $ is the standard scalar product in $%
\mathbb{E}^{4}.$

The second partial derivatives of $X(s,t)$ are expressed as follows%
\begin{eqnarray}
X_{ss} &=&a_{s}(s,t)V_{1}(s)+\left( \kappa _{1}(s)a(s,t)-\kappa
_{2}(s)b(s,t)\right) V_{2}(s)  \notag \\
&&+b_{s}(s,t)V_{3}(s)+\kappa _{3}(s)b(s,t)V_{4}(s),  \notag \\
X_{st} &=&a_{t}(s,t)V_{1}(s)+b_{t}(s,t)V_{3}(s),  \label{B5} \\
X_{tt} &=&A^{\prime \prime }(t)V_{2}(s)+B^{\prime \prime }(t)V_{4}(s)  \notag
\end{eqnarray}

Further, the normal space of $M$ is spanned by the vector fields%
\begin{eqnarray}
N_{1} &=&\frac{1}{\sqrt{(A^{\prime }(t))^{2}+(B^{\prime }(t))^{2}}}%
(-B^{\prime }V_{2}+A^{\prime }V_{4}),  \label{B6} \\
N_{2} &=&\frac{1}{\sqrt{a(s,t)^{2}+b(s,t)^{2}}}(-b(s,t)V_{1}+a(s,t)V_{3}). 
\notag
\end{eqnarray}%
\qquad Using (\ref{A5}), (\ref{B5}) and (\ref{B6}) we can calculate the
coefficients of the second fundamental form as follows:%
\begin{eqnarray}
c_{11}^{1} &=&\left \langle X_{ss}(s,t),N_{1}\right \rangle =\frac{A^{\prime
}b\kappa _{3}-B^{\prime }(\kappa _{1}a-\kappa _{2}b)}{\sqrt{(A^{\prime
}(t))^{2}+(B^{\prime }(t))^{2}}},\text{\ }  \notag \\
c_{12}^{1} &=&\left \langle X_{st}(s,t),N_{1}\right \rangle =0,  \notag \\
\text{ }c_{22}^{1} &=&\left \langle X_{tt}(s,t),N_{1}\right \rangle =\frac{%
A^{\prime }B^{\prime \prime }-B^{\prime }A^{\prime \prime }}{\sqrt{%
(A^{\prime }(t))^{2}+(B^{\prime }(t))^{2}}},  \label{B8} \\
c_{11}^{2} &=&\left \langle X_{ss}(s,t),N_{2}\right \rangle =\frac{%
ab_{s}-ba_{s}}{\sqrt{a(s,t)^{2}+b(s,t)^{2}}},  \notag \\
c_{12}^{2} &=&\left \langle X_{st}(s,t),N_{2}\right \rangle =\frac{%
ab_{t}-ba_{t}}{\sqrt{a(s,t)^{2}+b(s,t)^{2}}},  \notag \\
\text{ }c_{22}^{2} &=&\left \langle X_{tt}(s,t),N_{2}\right \rangle =0. 
\notag
\end{eqnarray}%
With the help of (\ref{A7}) and (\ref{B8}), we obtain the Gaussian curvature
given with the equation (\ref{B3}).
\end{proof}

An easy consequence of Proposition $1$ is the following.

\begin{corollary}
Let $M$ be a \textit{pencil surface} given by the parametrization (\ref{B1}%
). If 
\begin{equation}
A^{\prime }B^{\prime \prime }-B^{\prime }A^{\prime \prime }=0\text{ and }%
ab_{t}-ba_{t}=0,  \label{b9}
\end{equation}%
hold then $M$ has vanishing Gaussian curvature.
\end{corollary}

\begin{proposition}
Let $M$ be a \textit{pencil surface} given by the parametrization (\ref{B1}%
).Then the mean curvature of $M$ is%
\begin{equation}
\left \Vert \overrightarrow{H}\right \Vert ^{2}=\frac{1}{4}\left \{ \frac{%
(ab_{s}-ba_{s})^{3}}{E^{3}}+\frac{1}{G}\left[ \frac{\left( A^{\prime
}B^{\prime \prime }-B^{\prime }A^{\prime \prime }\right) }{G}+\frac{%
A^{\prime }b\kappa _{3}-B^{\prime }(\kappa _{1}a-\kappa _{2}b)}{E}\right]
^{2}\right \} .  \label{B9}
\end{equation}
\end{proposition}

\begin{proof}
Substituting (\ref{B4}) and (\ref{B8}) into (\ref{A9}) we get 
\begin{equation}
\overrightarrow{H}=\frac{1}{2EG^{3/2}}\left \{ E\left( A^{\prime }B^{\prime
\prime }\text{-}B^{\prime }A^{\prime \prime }\right) \text{+}G(A^{\prime
}b\kappa _{3}\text{-}B^{\prime }(\kappa _{1}a\text{-}\kappa _{2}b))\right \}
N_{1}\text{+}\frac{1}{2E^{3/2}}(ab_{s}\text{-}ba_{s})N_{2}.  \label{B10}
\end{equation}%
The norm of the mean curvature vector (\ref{B10}) gives (\ref{B9}).
\end{proof}

\begin{proposition}
Let $M$ be a pencil surface given by the parametrization (\ref{B1}). Then
the normal curvature of $M$ is%
\begin{equation}
K_{N}=\frac{(ab_{t}-ba_{t})}{E^{2}G^{2}}\left \{ G(A^{\prime }b\kappa
_{3}-B^{\prime }(\kappa _{1}a-\kappa _{2}b))-E\left( A^{\prime }B^{\prime
\prime }-B^{\prime }A^{\prime \prime }\right) \right \} .  \label{B11}
\end{equation}
\end{proposition}

\begin{proof}
Substituting (\ref{B4}) and (\ref{B8}) into (\ref{A8}) we get the result.
\end{proof}

Now, we consider some special cases of surface pencils.

\textbf{Case I: }Let $M$ be a pencil surface given with the $W$-curve 
\begin{equation*}
\gamma (s)=(a\cos cs,a\sin cs,b\cos ds,b\sin ds)
\end{equation*}%
as generator. Then $M$ becomes a \textit{generalized rotation surface} of
the form 
\begin{equation}
X(s,t)=(f(t)\cos cs,f(t)\sin cs,g(t)\cos ds,g(t)\sin ds)  \label{b12}
\end{equation}%
where%
\begin{equation}
\begin{array}{c}
f(t)=a+\frac{1}{k_{1}}(bd^{2}B(t)-ac^{2}A(t)), \\ 
g(t)=b+\frac{1}{k_{1}}(-bd^{2}A(t)-ac^{2}B(t)).%
\end{array}
\label{b13}
\end{equation}%
(see, \cite{GM}). For the case $f(t)=r(t)\cos t,$ $g(t)=r(t)\sin t$ and $%
c=d=1$ the generalized rotation surface $M$ is called \textit{Vranceanu
surface} \cite{Vr}. Furthermore, For the case $f(t)=\cos t,$ $g(t)=\sin t$
and $c\in \mathbb{R}^{+},d=1$ the generalized rotation surface $M$ is called 
\textit{Lawson surface \cite{La}.}\ 

\begin{proposition}
For the marching-scale functions \ 
\begin{eqnarray*}
A(t) &=&-\frac{ak_{1}(r(t)\cos t-a)+bk_{1}(r(t)\sin t-b)}{a^{2}+b^{2}}, \\
B(t) &=&\frac{bk_{1}(r(t)\cos t-a)-ak_{1}(r(t)\sin t-b)}{a^{2}+b^{2}}.
\end{eqnarray*}%
the pencil surface $M\subset \mathbb{E}^{4}$ becomes a Vranceanu surface$.$
\end{proposition}

\begin{proof}
Let $M$ be a pencil surface given by the parametrization (\ref{B1}).
Substituting $f(t)=r(t)\cos t,$ $g(t)=r(t)\sin t$ and $c=d=1$ into (\ref{b13}%
) we get the result.
\end{proof}

\begin{corollary}
Let $M$ be a \textit{pencil surface} given by the parametrization (\ref{B1}%
). If the smooth functions $A(t)$ and $B(t)$ are given by 
\begin{eqnarray*}
A(t) &=&-\frac{ak_{1}(\lambda e^{\mu t}\cos t-a)+bk_{1}(\lambda e^{\mu
t}\sin t-b)}{a^{2}+b^{2}}, \\
B(t) &=&\frac{bk_{1}(\lambda e^{\mu t}\cos t-a)-ak_{1}(\lambda e^{\mu t}\sin
t-b)}{a^{2}+b^{2}}.
\end{eqnarray*}%
then $M$ becomes a flat Vranceanu surface in $\mathbb{E}^{4}.$
\end{corollary}

\textbf{Case II: }Suppose $M$ is a pencil surface given with the
marching-scale functions $A(t)=B(t)=t.$

A \textit{standard ruled surface} $M$ in a 4-dimensional Euclidean space $%
\mathbb{E}^{4}$ is defined by 
\begin{equation}
M:\text{ }X(s,t)=\gamma (s)+t\beta (s),  \label{B.14}
\end{equation}%
where 
\begin{equation}
\beta (s)=\sum \limits_{i=2}^{4}\beta _{i}V_{i}(s),  \label{B.15}
\end{equation}%
is the unit vector in $\mathbb{E}^{4}$ and $V_{i}(s)$ the Frenet vector of
the unit speed curve $\gamma (s)$ \cite{Go}.

By the use of (\ref{B.14})- (\ref{B.15}) with (\ref{B1}) we get the
following result.

\begin{corollary}
Let $M$ be generalized standard ruled surface given with the parametrization%
\begin{equation}
X(s,t)=\gamma (s)+\frac{t}{\sqrt{2}}(V_{2}(s)+V_{4}(s)).  \label{B15*}
\end{equation}%
Then $M$ is a pencil surface with the marching scale functions $A(t)=B(t)=t.$
\end{corollary}

\begin{corollary}
Let $M$ be a pencil surface given with the parametrization (\ref{B15*}).
Then the Gaussian and normal curvatures of $M$ are given by 
\begin{equation*}
K=-\frac{(\kappa _{2}-\kappa _{3})^{2}}{\left( (1-t\kappa
_{1})^{2}+t^{2}(\kappa _{2}-\kappa _{3})^{2}\right) ^{2}},
\end{equation*}%
and%
\begin{equation*}
K_{N}=\frac{(\kappa _{2}-\kappa _{3})(t(\kappa _{1}^{2}+\kappa
_{2}^{2}-\kappa _{3}^{2})-\kappa _{1})}{2\left( (1-t\kappa
_{1})^{2}+t^{2}(\kappa _{2}-\kappa _{3})^{2}\right) ^{2}},
\end{equation*}%
respectively$.$
\end{corollary}

As a consequence of Corollary $8$ we obtain the following result.

\begin{corollary}
If $\kappa _{2}=\kappa _{3}$ then generalized standard ruled surface given
with the parametrization (\ref{B15*}) is a flat surface with flat normal
bundle.
\end{corollary}

\textbf{Case III: }Suppose $M$ is a pencil surface given the parametrization 
\begin{equation}
A(t)=r(t)\cos t,\text{ }B(t)=r(t)\sin t,  \label{b10}
\end{equation}

then we obtain the following result.

\begin{proposition}
Let $M$ be a \textit{pencil surface} given with the parametrization (\ref%
{b10}). If $M$ is a flat surface satisfying (\ref{b9}) then one of the
following case is occurs;

i) The profile curve $\gamma (s)$ is a planar and 
\begin{equation*}
r(t)=\frac{1}{c_{1}\sin t-c_{2}\cos t}.
\end{equation*}

ii) The profile curve $\gamma (s)\subset \mathbb{E}^{4}$ has curvatures $%
\kappa _{1},\kappa _{2}$ and $\kappa _{3}$ with 
\begin{equation*}
\kappa _{3}(s)=\frac{c_{1}\kappa _{2}(s)}{c_{2}+\kappa _{1}(s)}\text{ and }%
r(t)=\frac{1}{c_{1}\sin t-c_{2}\cos t}.
\end{equation*}

iii) The profile curve $\gamma (s)$ is a circle and 
\begin{equation*}
r(t)=-\frac{1}{c_{1}\sin t-c_{2}\cos t}.
\end{equation*}

iv) The profile curve $\gamma (s)$ has constant first curvature 
\begin{equation*}
\kappa _{1}(s)=\frac{1}{c_{1}}\text{ and }r(t)=\frac{c_{1}}{\cos t}.
\end{equation*}%
Here $c_{1}$ and $c_{2}$ are real constants.
\end{proposition}

\begin{proof}
Let $M$ be a pencil surface given with the parametrization (\ref{b10}). Then
substituting (\ref{b10}) into the (\ref{b9}) we get the following system of
differential equations 
\begin{eqnarray*}
2(r^{\prime })^{2}-rr^{\prime \prime }+r^{2} &=&0, \\
\kappa _{1}\kappa _{3}r^{2}+(r^{\prime }\kappa _{2}-r\kappa _{3})\cos
t-(r^{\prime }\kappa _{3}+r\kappa _{2})\sin t &=&0.
\end{eqnarray*}

Solving this system of equations with the help of Maple programme we get the
required results.
\end{proof}

\begin{tabular}{l}
Kadri Arslan, Bet\"{u}l Bulca \\ 
Department of Mathematics \\ 
Uluda\u{g} University \\ 
16059 Bursa, TURKEY \\ 
E-mails: arslan@uludag.edu.tr,\ bbulca@uludag.edu.tr\  \  \  \  \  \ 
\end{tabular}

\end{document}